\documentclass[a4paper,11pt]{amsart}

\usepackage{mathtools}
\usepackage{amsthm}
\usepackage{amssymb}
\usepackage[hidelinks]{hyperref}
\usepackage[nameinlink,noabbrev]{cleveref}

\newtheorem{theorem}{Theorem}[section]
\newtheorem{lemma}[theorem]{Lemma}

\theoremstyle{definition}
\newtheorem{definition}[theorem]{Definition}
\newtheorem{problem}[theorem]{Problem}
\newtheorem{counterexample}[theorem]{Counterexample}
\crefname{problem}{Problem}{Problems}
\Crefname{problem}{Problem}{Problems}
\crefname{counterexample}{Counterexample}{Counterexamples}
\Crefname{counterexample}{Counterexample}{Counterexamples}
\theoremstyle{remark}
\newtheorem{remark}[theorem]{Remark}

\newcommand{\Z}{\mathbb Z}
\newcommand{\Q}{\mathbb Q}
\newcommand{\Loag}{\mathcal L_{\mathrm{oag}}}

\title{A note on definable endomorphisms of ordered abelian groups}
\author{Haitao Ji, Fengshuo Xu, and Dong Qiu$^{*}$}\thanks{Corresponding author. E-mail: qiudong@gux.edu.cn (D. Qiu).}

\date{}

\begin{document}
\maketitle

\begin{center}
\small
School of Mathematics, Guangxi University, Nanning 530004, China
\end{center}

\begin{abstract}
We answer Kourovka Notebook Problem~18.16 affirmatively in the pure language of ordered abelian groups, with parameters allowed.  The piecewise-affine description of definable functions reduces the question to an algebraic rigidity theorem: an additive endomorphism of a torsion-free abelian group that is covered by finitely many rational-affine laws has one global rational slope.    The parameter-free pure-language case is included, whereas the unrestricted expansion-language variant admits a simple counterexample.
\end{abstract}



\section{Introduction}

O.~V.~Belegradek posed the following question as Problem~18.16 in the Kourovka Notebook~\textup{\cite{kourovka21}}:
\begin{problem}\label{prob:1816}
\textup{\cite{kourovka21}} Is every definable endomorphism of every ordered abelian group of the form $x\mapsto rx$ for some $r\in\Q$?
\end{problem}

The problem remains listed without a solution comment in the current edition.  We use the standard convention that ``definable'' means definable with parameters unless stated otherwise~\textup{\cite{dolichGoodrick2017}}, and we work in the pure language $\Loag=\{0,+,<\}$.  Thus the parameter-free pure-language case is included by taking the parameter set to be empty.  This language restriction is essential; \cref{ctr:arbitrary-expansion} shows that the unrestricted expansion-language variant is false.

Cluckers and Halupczok proved that every definable map $G^n\to G$ in an ordered abelian group is rational-affine on the members of a finite definable partition~\textup{\cite{cluckersHalupczok2011}}.  For groups of finite regular rank, piecewise linearity had previously been obtained by Belegradek, Verbovskiy, and Wagner~\textup{\cite{belegradekVerbovskiyWagner2003}}.  Definable endomorphisms in expansions of ordered groups belong to a broader setting studied, for example, by Miller~\textup{\cite{miller2001}}.  The pure-language piecewise-linearity result does not by itself say that the finitely many affine laws of an additive endomorphism have the same slope.

The deduction is completed by an elementary algebraic step, for which we give a self-contained proof: a finite rational-affine cover of an endomorphism on any torsion-free abelian group collapses to one rational-linear law.  First, applying the cover to sufficiently many integer multiples of one element removes the affine offset at that element.  Applying the resulting pointwise statement along an affine line then gives one cover slope simultaneously at any prescribed pair of elements.  A nonzero base point makes that slope unique.  Since every ordered abelian group is torsion-free, the piecewise-affine theorem then answers the Kourovka problem.

\section{Finite affine covers and global slopes}

We write the group operation additively and use $kx$ for the integer multiple of $x$ by $k\in\Z$.

\begin{definition}\label{def:finite-affine-cover}
\textup{\cite{cluckersHalupczok2011}} Let $G$ be an abelian group and let $f\colon G\to G$ be an endomorphism.  A \emph{finite rational-affine cover} of $f$ consists of a finite nonempty index set $I$, subsets $(P_i)_{i\in I}$ covering $G$, integers $a_i,d_i\in\Z$ with $d_i\ne0$, and elements $b_i\in G$ such that
\begin{equation}\label{eq:affine-piece}
 d_i f(x)=a_i x+b_i\qquad(x\in P_i).
\end{equation}

\end{definition}
Relative to the finite definable partition in the piecewise-affine theorem, this definition retains only the affine identities: the sets $P_i$ need not be definable and need only cover the group rather than form a partition. 
\begin{lemma}\label{lem:pointwise-slope}
Let $G$ be torsion-free, and suppose that an endomorphism $f\colon G\to G$ has a finite rational-affine cover.  For every $x\in G$, there is an $i\in I$ such that
\begin{equation}\label{eq:pointwise-slope}
 d_i f(x)=a_i x.
\end{equation}
\end{lemma}

\begin{proof}
Write $N=|I|$.  Among the $N+1$ elements $0x,x,\ldots,Nx$, two, say $px$ and $qx$ with $0\le p<q\le N$, belong to the same set $P_i$.  Applying \eqref{eq:affine-piece} at these two points and using additivity gives
\[
 p\,d_i f(x)=p\,a_i x+b_i,
 \qquad
 q\,d_i f(x)=q\,a_i x+b_i.
\]
Subtracting yields
\[
 (q-p)\bigl(d_i f(x)-a_i x\bigr)=0.
\]
Since $q-p\ne0$ and $G$ is torsion-free, \eqref{eq:pointwise-slope} follows.
\end{proof}
\Cref{lem:pointwise-slope} shows that the offsets in \eqref{eq:affine-piece} disappear after choosing a piece that may depend on the point.
\begin{lemma}\label{lem:common-slope}
Under the hypotheses of \cref{lem:pointwise-slope}, for every $x,y\in G$ there is an $i\in I$ such that
\[
 d_i f(x)=a_i x
 \quad\text{and}\quad
 d_i f(y)=a_i y.
\]
\end{lemma}

\begin{proof}
For each $k\in\{0,\ldots,|I|\}$, apply \cref{lem:pointwise-slope} to $x+ky$ and choose an index $i_k$ for which
\begin{equation}\label{eq:line-slope}
 d_{i_k}f(x+ky)=a_{i_k}(x+ky).
\end{equation}
The indices $i_0,\ldots,i_{|I|}$ all belong to $I$, so two of them coincide; thus $i_p=i_q=i$ for some $p<q$.  Subtracting the two instances of \eqref{eq:line-slope} and using additivity gives
\[
 (q-p)\bigl(d_i f(y)-a_i y\bigr)=0.
\]
Torsion-freeness gives $d_i f(y)=a_i y$.  Substitution into the instance of \eqref{eq:line-slope} for $p$ then gives $d_i f(x)=a_i x$.
\end{proof}

\begin{theorem}\label{thm:collapse}
Let $G$ be a torsion-free abelian group and let $f\colon G\to G$ be an endomorphism with a finite rational-affine cover.  Then there are integers $a,d\in\Z$, with $d\ne0$, such that
\begin{equation}\label{eq:global-slope}
 d f(x)=a x\qquad\text{for every }x\in G.
\end{equation}
\end{theorem}

\begin{proof}
If $G=0$, take $a=0$ and $d=1$.  Otherwise choose $x_0\ne0$.  By \cref{lem:pointwise-slope}, there is an $i_0\in I$ such that
\[
 d_{i_0}f(x_0)=a_{i_0}x_0.
\]
Fix $y\in G$.  By \cref{lem:common-slope}, some $i\in I$ satisfies
\[
 d_i f(x_0)=a_i x_0,
 \qquad
 d_i f(y)=a_i y.
\]
Cross-multiplying the two equations at $x_0$ gives
\[
 (d_i a_{i_0}-d_{i_0}a_i)x_0=0.
\]
Since $x_0\ne0$ and $G$ is torsion-free,
\begin{equation}\label{eq:cross-slope}
 d_i a_{i_0}=d_{i_0}a_i.
\end{equation}
Consequently,
\[
 d_i\bigl(d_{i_0}f(y)-a_{i_0}y\bigr)
 =d_{i_0}d_i f(y)-d_i a_{i_0}y
 =d_{i_0}a_i y-d_i a_{i_0}y=0,
\]
where the last equality uses \eqref{eq:cross-slope}.  As $d_i\ne0$, torsion-freeness yields $d_{i_0}f(y)=a_{i_0}y$.  Since $y$ was arbitrary, \eqref{eq:global-slope} holds with $a=a_{i_0}$ and $d=d_{i_0}$.
\end{proof}
Thus $f$ is multiplication by the rational number $a/d$ after embedding $G$ in its divisible hull~\textup{\cite{fuchs1970}}.
\section{The definable endomorphism problem}

\begin{lemma}\label{thm:piecewise-affine}
\textup{\cite{cluckersHalupczok2011}} If $G$ is an ordered abelian group and $h\colon G^n\to G$ is $\Loag$-definable with parameters, then $G^n$ has a finite definable partition such that on each part
\[
 h(x_1,\ldots,x_n)=\frac{1}{s}\left(r_1x_1+\cdots+r_nx_n+b\right)
\]
in the sense that $s h(x_1,\ldots,x_n)=r_1x_1+\cdots+r_nx_n+b$, for suitable integers $r_1,\ldots,r_n,s$ with $s\ne0$ and a suitable $b\in G$.
\end{lemma}


\begin{theorem}\label{cor:1816}
 If $G$ is an ordered abelian group and $f\colon G\to G$ is a definable endomorphism, then there are $a,d\in\Z$, with $d\ne0$, such that
\[
 d f(x)=a x\qquad\text{for every }x\in G.
\]
\end{theorem}

\begin{proof}
Apply \cref{thm:piecewise-affine} with $n=1$ to the definable function $f\colon G\to G$.  Its finite definable partition supplies a finite rational-affine cover in the sense of \cref{def:finite-affine-cover}.  Every ordered abelian group is torsion-free, so \cref{thm:collapse} gives the asserted integers $a$ and $d$.  Taking the parameter set to be empty gives the parameter-free case.
\end{proof}
Thus, the answer to \cref{prob:1816} is affirmative for definability in $\Loag$, with parameters allowed.  In particular, it is affirmative for parameter-free definability.  
\begin{counterexample}\label{ctr:arbitrary-expansion}
The analogous assertion for arbitrary expansions of the language is false.  Expand $(\mathbb R,0,+,<)$ by a unary function symbol $\lambda$ interpreted by $\lambda(x)=\sqrt2\,x$.  Then $\lambda$ is a parameter-free definable, order-preserving additive endomorphism of the expansion, but it is not multiplication by a rational number.

The reasons are as follows.
The graph of $\lambda$ is parameter-free definable because $\lambda$ is named in the expanded language.  Additivity is immediate, and $\sqrt2>0$ makes $\lambda$ order-preserving.  If $\lambda(x)=rx$ for every $x\in\mathbb R$ and some $r\in\mathbb Q$, then evaluation at $x=1$ gives $r=\sqrt2$, contradicting the irrationality of $\sqrt2$.
\end{counterexample}

\begin{remark}\label{rem:scope}
 In the pure language, parameter-free definability is a special case of definability with parameters, and both satisfy \cref{cor:1816}; arbitrary language expansions fail by \cref{ctr:arbitrary-expansion}.  For a nondivisible $G$, the equation $d f(x)=a x$ is the intrinsic formulation; it says that $f$ agrees with multiplication by $a/d$ in the divisible hull of $G$~\textup{\cite{fuchs1970}}.
\end{remark}
\section*{Acknowledgements}
We employed Canto, our self-built intelligent agent, which provided key references [2]. 
This work was supported by the National Natural Science Foundation of China (Grant No.~12571489) and Guangxi Natural Science Foundation (No.~2025GXNS-FAA069576).
\begingroup
\small
\bibliographystyle{plain}
\bibliography{references}
\endgroup

\end{document}